\documentclass[10pt]{article}
\usepackage{amsmath}
\usepackage{amsfonts}
\usepackage{amssymb}
\usepackage{graphicx, times}
\usepackage{subfig}
\usepackage{amsthm}
\usepackage{bm}
\usepackage[section]{placeins}  
\newtheorem{The}{Theorem}[section]
\newtheorem{Def}{Definition}[section]

\newtheorem{Ex}{Example}[section]
\newtheorem{remark}{Remark}[section]

\begin{document}
\begin{center}
{\LARGE {\bf Can we split fractional derivative while analyzing fractional differential equations? }}
\vskip 1cm
{\Large  Sachin Bhalekar, Madhuri Patil}\\
\textit{Department of Mathematics, Shivaji University, Kolhapur - 416004, India, Email:sachin.math@yahoo.co.in, sbb\_maths@unishivaji.ac.in (Sachin Bhalekar), madhuripatil4246@gmail.com (Madhuri Patil)}\\
\end{center}
\begin{abstract}
 Fractional derivatives are generalization to classical integer-order derivatives. The rules which are true for classical derivative need not hold for the fractional derivatives, for example, we cannot simply add the fractional orders $\alpha$ and $\beta$ in ${}_0^{C}\mathrm{D}_t^\alpha {}_0^{C}\mathrm{D}_t^\beta$ to produce the fractional derivative ${}_0^{C}\mathrm{D}_t^{\alpha+\beta}$ of order $\alpha+\beta$, in general. In this article we discuss the details of such compositions and propose the conditions to split a linear fractional differential equation into the systems involving lower order derivatives. Further, we provide some examples, which show that the related results in the literature are sufficient but not necessary conditions.
\end{abstract}
\section{Introduction}
Fractional Calculus (FC) is a popular branch of Mathematics which has attracted the researchers working in various fields of Science, Engineering and Social Sciences\cite{Podlubny, Mainardi, Magin}. The ability of the fractional derivatives (FDs) to model memory properties in the real-life models is a key to the applicability of fractional differential equations (FDEs). 
\par Applications of FC in viscoelasticity are given in \cite{Koeller,Shimizu,Craiem}. In \cite{Kulish}, Kulish and Lage presented the application of FC to the solution of time-dependent, viscous-diffusion fluid mechanics problems. Fellah et al. \cite{Fellah} used FC to model the sound waves propagation in rigid
porous materials. In \cite{Magin}, Magin described the applications of FC to solve biomedical problems. Sebaa et al. \cite{Sebaa} used FC to describe the viscous interactions between fluid and solid structure in cancellous bone. Fractional derivatives are widely used in control theory \cite{Matignon,Matusu,Baleanu}. In the book \cite{Fallahgoul}, Fallahgoul et al. discussed how FC and fractional processes are used in financial modeling, finance and economics. In \cite{Goulart}, Goulart et al. proposed two fractional differential equation models for the spatial distribution of concentration of a non-reactive pollutants in planetary boundary layer.  
\par The FDs are so flexible that the order can be chosen not only from the set of positive integers but also from real and complex number sets \cite{Love,Neamaty}. Surprisingly, the order of FD can also be a function of time \cite{Samko-1995,Valerio} or may distributed on some interval \cite{Caputo, Bagley, Refahi}.
\par However, one has to be careful while using FDE models. Due to the generalization, the FD becomes nonlocal unlike classical derivative. Hence, the properties and rules which are trivial for classical derivatives (e.g. chain rule, Leibniz rule) become complicated with FDs. The FDE models may also behave weirdly. e.g. The trajectories of autonomous planar systems involving classical derivative cannot have singular points but the fractional order counterparts of the same model can have self-intersecting trajectories, cusps etc. \cite{S. Bhalekar}.
\par For any positive integers $m$ and $n$, we have
$$\frac{d^n}{dx^n}\frac{d^m}{dx^m}(\cdot)=\frac{d^m}{dx^m}\frac{d^n}{dx^n}(\cdot)=\frac{d^{m+n}}{dx^{m+n}}(\cdot).$$
On the other hand, if we replace the integers $m$ and $n$ by arbitrary numbers then the resulting FDs need not hold such composition rule. Ordinary differential equation  
$
a_0x(t)+a_1\frac{d}{dt}x(t)+a_2\frac{d^2}{dt^2}x(t)+\dots+ a_n\frac{d^n}{dt^n}x(t)=0, 
$
of higher order can be splitted into a system
\begin{equation*}
\begin{split}
x(t)& = y_0(t),\\
\frac{d}{dt}y_j(t)& = y_{j+1}(t),\, j=0,1,\dots,n-2,\\
\frac{d}{dt}y_{n-1}(t)& = \frac{-1}{a_n}(a_0y_0(t)+a_1y_1(t)+a_2y_2(t)+\dots+a_{n-1}y_{n-1}(t))
\end{split}
\end{equation*} 
containing lower order derivatives. This is not the case with FDEs, in general. 
\par In this article we propose the results regarding such compositions of FDs and splitting of FDEs. 
\section{Preliminaries}
This section deals with basic definitions and results given in the literature \cite{Podlubny,Das,  Samko, Erdelyi, Diethelm}. Throughout this section, we take $n\in\mathbb{N}$.
\begin{Def} \label{Def 2.1}
Let $\alpha\ge0$ \,\, ($\alpha\in\mathbb{R}$). Then Riemann-Liouville (\text RL) fractional integral of function $f\in C[a,b]$, $t>a$ of order `$\alpha$' is defined as,
\begin{equation}
{}_a\mathrm{I}_t^\alpha f(t)=
\frac{1}{\Gamma{(\alpha)}}\int_a^t (t-\tau)^{\alpha-1}f(\tau)\,\mathrm{d}\tau. \label{2.1}
\end{equation}
\end{Def}
\begin{Def}\label{Def 2.2}
The Caputo fractional derivative of order $\alpha>0$, $n-1<\alpha< n$, $n\in \mathbb{N}$ is defined for $f\in C^n[a,b]$,\, $t>a$ as,
\begin{equation}
{}_a^{C}\mathrm{D}_t^\alpha f(t)=
\begin{cases}
\frac{1}{\Gamma{(n-\alpha)}}\int_a^t (t-\tau)^{n-\alpha-1}f^{(n)}(\tau)\,\mathrm{d}\tau & \mathrm{if}\,\, n-1<\alpha< n\\
\frac{d^n}{dt^n}f(t) & \mathrm{if}\,\, \alpha=n.
\end{cases}\label{2.2}
\end{equation}
Note that ${}_0^{C}\mathrm{D}_t^\alpha c=0$, where $c$ is a constant.
\end{Def}
\begin{Def} \label{Def 2.3}
The one-parameter Mittag-Leffler function is defined as,
\begin{equation}
E_\alpha(z)=\sum_{k=0}^\infty \frac{z^k}{\Gamma(\alpha k+1)}\, ,\qquad z\in\mathbb{C} \,(\alpha>0).\label{2.3}
\end{equation}
The two-parameter Mittag-Leffler function is defined as,
\begin{equation}
E_{\alpha,\beta}(z)=\sum_{k=0}^\infty \frac{z^k}{\Gamma(\alpha k+\beta)}\, ,\qquad z\in\mathbb{C} \,(\alpha>0,\,\beta>0).\label{2.4}
\end{equation}
\end{Def}
\begin{Def} \cite{Podlubny} \label{Def 2.4}
The multi-parameter Mittag-Leffler function is defined as,
\begin{equation}
E_{(a_1,a_2,\dots,a_n),b}(z_1,\dots,z_n)=\sum_{k=0}^\infty\sum_{\substack{l_1+\dots+l_n=k\\l_1\ge0,\dots,l_n\ge0}}(k;l_1,\dots,l_n) \frac{\prod_{i=1}^{n}z_i^{l_i}}{\Gamma(b+\sum_{i=1}^{n}a_il_i)} \label{2.5}
\end{equation}
where, $(k;l_1,\dots,l_n)=\frac{k!}{l_1!l_2!\cdots l_n!}$ is a multinomial coefficient.
\end{Def}
\begin{Def}\cite{Gorenflo} \label{Def 2.5}
The Prabhakar generalized Mittag-Leffler function is defined as,
\begin{equation}
E_{\alpha,\beta}^\gamma(z)=\sum_{k=0}^\infty \frac{(\gamma)_kz^k}{\Gamma(\alpha k+\beta)k!}\, ,\qquad z\in \mathbb{C}\quad (Re(\alpha)>0,\,Re(\beta)>0,\,\gamma>0).\label{2.6}
\end{equation}
where $(\gamma)_k=\frac{\Gamma(\gamma+k)}{\Gamma(\gamma)}=\gamma(\gamma+1)\cdots(\gamma+k-1)$.
\end{Def}
\noindent{\bf properties:}\cite{Podlubny} 
If $\mathcal{L}\left\{f(t);s\right\}=F(s)$ is the Laplace transform of function $f$ then 
\begin{equation}
\mathcal{L}\left\{{}_0^C\mathrm{D}_t^\alpha f(t);s\right\}=s^\alpha F(s)-\sum_{k=0}^{n-1} s^{\alpha-k-1} f^{(k)}(0),\qquad(n-1<\alpha\le n). \qquad \label{2.7}
\end{equation}
\\
{\bf Note that}\\
(i) Let $n-1< \alpha \le n$, and $\beta\ge0$\\
\quad ${}_0^{C}\mathrm{D}_t^\alpha t^\beta=
\begin{cases}
\frac{\Gamma(\beta+1)}{\Gamma(-\alpha+\beta+1)}t^{\beta-\alpha} ,\, \mathrm{if}\, \beta>n-1, \, \beta\in\mathbb{R} \\
0\qquad \qquad \quad \quad\, ,\, \mathrm{if}\, \beta\in \{0,1,2,\dots,n-1\}.
\end{cases}$\\
(ii) ${}_0^{C}\mathrm{D}_t^\alpha {}_0\mathrm{I}_t^\beta f(t)=
\begin{cases} 
{}_0\mathrm{I}_t^{\beta-\alpha }f(t), \quad \mathrm{if} \,   \beta>\alpha.\\
f(t) , \qquad \quad \,\,\,  \mathrm{if} \, \beta=\alpha.\\
{}_0^{C}\mathrm{D}_t^{\alpha-\beta} f(t), \quad \mathrm{if} \, \alpha> \beta.
\end{cases}$\\
\begin{The} \label{Thm 2.1}
\cite{Luchko} Solution of homogeneous fractional order differential equation
\begin{equation}
{}_0^C\mathrm{D}_t^\alpha x(t)+\lambda x(t)=0, \qquad 0<\alpha<1 \label{2.8}
\end{equation} 
is given by,
\begin{equation}
x(t)=x(0)E_\alpha(-\lambda t^\alpha).\label{2.9}
\end{equation}
\begin{The}\label{Thm 2.2} \cite{Podlubny}                                                                         
Let $f(t)$ is continuous for $t\ge 0$ then 
$${}_0\mathrm{I}_t^\alpha {}_0\mathrm{I}_t^\beta f(t)={}_0\mathrm{I}_t^\beta {}_0\mathrm{I}_t^\alpha f(t)= {}_0\mathrm{I}_t^{\alpha+\beta} f(t)$$  for any $\alpha>0$ , $\beta>0$.
\end{The}
\begin{The}\cite{Podlubny}\label{Thm 2.3}
Let $n-1<\alpha<n$,\, $n\in\mathbb{N}$ and $m\in\mathbb{N}$. If $f\in C^{m+n}[a,b]$ then
$${}_0^{C}\mathrm{D}_t^\alpha {}_0^{C}\mathrm{D}_t^mf(t)={}_0^{C}\mathrm{D}_t^{\alpha+m}f(t)\ne {}_0^{C}\mathrm{D}_t^m {}_0^{C}\mathrm{D}_t^\alpha f(t).$$ 
\end{The}
\end{The}
\begin{The} \cite{Li and Deng} \label{Thm 2.4}
If $x(t)\in C^m[0,T]$ for $T>0$ and $m-1<\alpha<m, m\in \mathbb{N}$. Then ${}_0^C\mathrm{D}_t^\alpha x(0)=0.$
\end{The}
\begin{The}\cite{Li and Deng} \label{Thm 2.5}
If $x(t)\in C^1[0,T]$ for some $T>0$, then $${}_0^C\mathrm{D}_t^{\alpha_2} {}_0^C\mathrm{D}_t^{\alpha_1} x(t)={}_0^C\mathrm{D}_t^{\alpha_1}{}_0^C\mathrm{D}_t^{\alpha_2} x(t)={}_0^C\mathrm{D}_t^{\alpha_1+\alpha_2}x(t), \quad t\in [0,T].$$
where $\alpha_1, \alpha_2 \in \mathbb{R}^+$ and $\alpha_1+\alpha_2\le1$.
\end{The}
\begin{The}\cite{Diethelm} \label{Thm 2.6}
Let $f \in C^k[0,b]$ for some $b>0$ and some $k\in\mathbb{N}$. Moreover let $\alpha, \varepsilon >0$ be such that there exists some $l\in\mathbb{N}$ with $l\le k$ and $\alpha, \alpha+\varepsilon \in [l-1,l]$. Then,
$$
{}_0^C\mathrm{D}_t^{\varepsilon}{}_0^C\mathrm{D}_t^{\alpha} f(t)={}_0^C\mathrm{D}_t^{\alpha+\varepsilon}f(t).
$$
\end{The}
\section{Results on Laplace transform of Mittag-Lffler function}
\begin{The}\label{Thm 2.7}
The Laplace transform of $f(t)=t^{b-1}E_{(\alpha-\beta,\alpha),b}(-a_1t^{\alpha-\beta},-a_2t^\alpha)$ is given by,
\begin{equation}
\mathcal{L}\left\{t^{b-1}E_{(\alpha-\beta,\alpha),b}(-a_1t^{\alpha-\beta},-a_2t^\alpha);s\right\}=\frac{s^{\alpha-b}}{s^\alpha+a_1s^\beta+a_2} \label{2.10}
\end{equation}
where, $b\ge 1$ and $0<\beta<\alpha$.
\end{The}
\begin{proof}
\begin{equation*}
\begin{split}
\mathcal{L}\left\{t^{b-1}E_{(\alpha-\beta,\alpha),b}(-a_1t^{\alpha-\beta},-a_2t^\alpha);s\right\}& = \int_0^\infty e^{-st}\left(\sum_{k=0}^\infty\sum_{l_1+l_2=k}\frac{k!}{l_1!l_2!} \frac{(-a_1)^{l_1}(-a_2)^{l_2}t^{(\alpha-\beta)l_1+\alpha l_2+b-1}}{\Gamma(b+(\alpha-\beta)l_1+\alpha l_2)}\right) \\
& = \sum_{k=0}^\infty\sum_{l_1+l_2=k}\frac{k!}{l_1!l_2!}\frac{(-a_1)^{l_1}(-a_2)^{l_2}}{S^{b+(\alpha-\beta)l_1+\alpha l_2}}\\
& = s^{-b}\sum_{k=0}^\infty\sum_{l_1=0}^k \binom{k}{l_1}(-a_1s^{-\alpha+\beta})^{l_1}(-a_2s^{-\alpha})^{k-l_1}\\
& = s^{-b}\sum_{k=0}^\infty(-1)^k(a_1s^{-\alpha+\beta}+a_2s^{-\alpha})^k\\
& = \frac{s^{\alpha-b}}{s^\alpha+a_1s^\beta+a_2}.
\end{split}
\end{equation*}
\end{proof}
\begin{remark}
(i)\,If $0<\alpha_3<\alpha_2<\alpha_1$ and $b\ge 1$ then $$\mathcal{L}\left\{t^{b-1}E_{(\alpha_1-\alpha_2,\alpha_1-\alpha_3,\alpha_1),b}(-a_1t^{\alpha_1-\alpha_2},-a_2t^{\alpha_1-\alpha_3},-a_3t^{\alpha_1});s\right\}=\frac{s^{\alpha_1-b}}{s^{\alpha_1}+a_1s^{\alpha_2}+a_2s^{\alpha_3}+a_3}.$$\\[0.1cm]
(ii)\,If $0<\alpha_4<\alpha_3<\alpha_2<\alpha_1$ and $b\ge 1$ then 
\begin{multline*}
\mathcal{L}\left\{t^{b-1}E_{(\alpha_1-\alpha_2,\alpha_1-\alpha_3,\alpha_1-\alpha_4,\alpha_1),b}(-a_1t^{\alpha_1-\alpha_2},-a_2t^{\alpha_1-\alpha_3}, -a_3t^{\alpha_1-\alpha_4},-a_4t^{\alpha_1});s\right\}=\\
\frac{s^{\alpha_1-b}}{s^{\alpha_1}+a_1s^{\alpha_2}+a_2s^{\alpha_3}+a_3s^{\alpha_4}+a_4}.
\end{multline*}
\end{remark}
\section{Can we use composition rule while analyzing fractional differential equations?}
This section deals with the results regarding the composition of fractional derivatives.
\begin{The} \label{Thm 2.8}
The solution of 
\begin{equation}
{}_0^C\mathrm{D}_t^\alpha x(t)=\lambda x(t), \quad 0<\alpha<1 \label{2.11}
\end{equation}
is not $C^1$ at $t=0$, where $\lambda\in\mathbb{R}$, $\lambda\ne0$.
\end{The}
\begin{proof}
Using Theorem \ref{Thm 2.1} the solution of equation (\ref{2.11}) is given by 
\begin{equation}
x(t)=x(0)E_\alpha(\lambda t^\alpha). \label{2.12}
\end{equation} 
By using definition \ref{Def 2.3}, we have 
\begin{equation*}
\begin{split}
E_\alpha(\lambda t^\alpha) & = \sum_{k=0}^\infty \frac{(\lambda t^\alpha)^k}{\Gamma(\alpha k+1)}\\
& =  1+\frac{\lambda t^\alpha}{\Gamma(\alpha +1)}+\frac{\lambda^2 t^{2\alpha}}{\Gamma(2\alpha +1)}+\cdots.
\end{split}
\end{equation*}
\begin{equation*}
\therefore \frac{d}{dt}E_\alpha(\lambda t^\alpha)=\frac{\lambda t^{\alpha-1}}{\Gamma(\alpha )}+\frac{\lambda^2 t^{2\alpha-1}}{\Gamma(2\alpha )}+\frac{\lambda^3 t^{3\alpha-1}}{\Gamma(3\alpha )}+\cdots.
\end{equation*}
Since $\alpha-1<0$, \, $t^{\alpha-1}$ is not continuous at $t=0$.\\
$\Rightarrow$ $\frac{d}{dt}E_\alpha(\lambda t^\alpha)$ is not continuous in $[0,T]$ for any $T>0$.\\
$\Rightarrow$ $ E_\alpha(\lambda t^\alpha)\notin C^1[0,T]$.\\
i.e. The solution $x(t)$ given by equation (\ref{2.12}) is not $C^1$ at $t=0$.
\end{proof}
{\bf Note:}Theorem \ref{Thm 2.8} shows that the solution of FDE (\ref{2.11}) does not satisfy the condition in Theorem \ref{Thm 2.5}.
Further, the following example shows that  \,  ${}_0^C\mathrm{D}_t^{\alpha_1} {}_0^C\mathrm{D}_t^{\alpha_2}E_\alpha(\lambda t^\alpha)\ne  {}_0^C\mathrm{D}_t^{\alpha_1+\alpha_2}E_\alpha(\lambda t^\alpha)$, in general. 
\begin{Ex}
\begin{equation}
{}_0^C\mathrm{D}_t^{\frac{1}{4}} {}_0^C\mathrm{D}_t^{\frac{1}{4}}E_\frac{1}{4}(\lambda t^\frac{1}{4})= \lambda^2 E_\frac{1}{4}(\lambda t^\frac{1}{4}) \label{2.13}
\end{equation}
and
\begin{equation}
{}_0^C\mathrm{D}_t^{\frac{1}{2}}E_\frac{1}{4}(\lambda t^\frac{1}{4})= \lambda t^{-\frac{1}{4}} E_{\frac{1}{4},\frac{3}{4}}(\lambda t^\frac{1}{4}) \label{2.14}
\end{equation}
From (\ref{2.13}) and (\ref{2.14}) we can see that
\begin{equation*}
{}_0^C\mathrm{D}_t^{\frac{1}{4}} {}_0^C\mathrm{D}_t^{\frac{1}{4}}E_\frac{1}{4}(\lambda t^\frac{1}{4})\ne {}_0^C\mathrm{D}_t^{\frac{1}{2}}E_\frac{1}{4}(\lambda t^\frac{1}{4}).
\end{equation*}
\end{Ex}
Therefore, The natural question is ``Can we split FDE (\ref{2.11}) into a system of lower order FDEs?" We provide the answer to this question in subsequent sections.
\subsection{Composition rule for Mittag-Leffler function}
\begin{The} \label{Thm 2.9}
If $\alpha_1>0$, $\alpha_2>0$ are real numbers satisfying $0<\alpha_1+\alpha_2\le\alpha<1$ then\begin{equation*}
{}_0^C\mathrm{D}_t^{\alpha_1} {}_0^C\mathrm{D}_t^{\alpha_2}E_\alpha(\lambda t^\alpha)={}_0^C\mathrm{D}_t^{\alpha_1+\alpha_2}E_\alpha(\lambda t^\alpha)={}_0^C\mathrm{D}_t^{\alpha_2}{}_0^C\mathrm{D}_t^{\alpha_1} E_\alpha(\lambda t^\alpha)
\end{equation*}
even though $E_\alpha(\lambda t^\alpha)$ is not $C^1$ at $t=0$.
\begin{proof}
\begin{equation*}
\begin{split}
{}_0^C\mathrm{D}_t^{\alpha_1} {}_0^C\mathrm{D}_t^{\alpha_2}E_\alpha(\lambda t^\alpha) & = {}_0^C\mathrm{D}_t^{\alpha_1}\left\{{}_0^C\mathrm{D}_t^{\alpha_2}\sum_{k=0}^\infty \frac{(\lambda t^\alpha)^k}{\Gamma(\alpha k+1)}\right\}\\
& = {}_0^C\mathrm{D}_t^{\alpha_1}\left\{\sum_{k=1}^\infty \frac{\lambda^k t^{\alpha k-\alpha_2}}{\Gamma(\alpha k-\alpha_2+1)}\right\}\\
& = \sum_{k=1}^\infty \frac{\lambda^k t^{\alpha k-\alpha_2-\alpha_1}}{\Gamma(\alpha k-\alpha_2-\alpha_1+1)}\\
& = \sum_{k=0}^\infty \frac{\lambda^{k+1} t^{\alpha k+\alpha-\alpha_2-\alpha_1}}{\Gamma(\alpha k+\alpha-\alpha_2-\alpha_1+1)}\\
& = \lambda t^{\alpha-\alpha_2-\alpha_1}E_{\alpha-\alpha_2-\alpha_1+1}(\lambda t^\alpha). 
\end{split}
\end{equation*}
In a similar manner, we find \, \begin{equation*}
{}_0^C\mathrm{D}_t^{\alpha_2} {}_0^C\mathrm{D}_t^{\alpha_1}E_\alpha(\lambda t^\alpha)=\lambda t^{\alpha-\alpha_2-\alpha_1}E_{\alpha-\alpha_2-\alpha_1+1}(\lambda t^\alpha)
\end{equation*}
and 
\begin{equation*}
\begin{split}
{}_0^C\mathrm{D}_t^{\alpha_1+\alpha_2}E_\alpha(\lambda t^\alpha) & = {}_0^C\mathrm{D}_t^{\alpha_1+\alpha_2}\left\{\sum_{k=0}^\infty \frac{(\lambda t^\alpha)^k}{\Gamma(\alpha k+1)}\right\}\\
& = \sum_{k=1}^\infty \frac{\lambda^k t^{\alpha k-\alpha_1-\alpha_2}}{\Gamma(\alpha k-\alpha_1-\alpha_2+1)}\\
& = \sum_{k=0}^\infty \frac{\lambda^{k+1} t^{\alpha k+\alpha-\alpha_1-\alpha_2}}{\Gamma(\alpha k+\alpha-\alpha_1-\alpha_2+1)}\\
& = \lambda t^{\alpha-\alpha_1-\alpha_2}E_{\alpha-\alpha_1-\alpha_2+1}(\lambda t^\alpha). 
\end{split}
\end{equation*} 
Thus, ${}_0^C\mathrm{D}_t^{\alpha_1} {}_0^C\mathrm{D}_t^{\alpha_2}E_\alpha(\lambda t^\alpha)={}_0^C\mathrm{D}_t^{\alpha_1+\alpha_2}E_\alpha(\lambda t^\alpha) = {}_0^C\mathrm{D}_t^{\alpha_2} {}_0^C\mathrm{D}_t^{\alpha_1} E_\alpha(\lambda t^\alpha)$ \,if \\ $0<\alpha_1+\alpha_2\le\alpha<1$.
\end{proof}
\end{The}
{\bf Note 1:}\\
  This Theorem \ref{Thm 2.9} shows that the condition given in Theorem \ref{Thm 2.5} is sufficient but not necessary.\\
 
\begin{remark}
If we consider $ \alpha_1+\alpha_2>\alpha$ and $0<\alpha<1$ then above theorem does not holds.
We can verify this from following examples.
\begin{Ex}
We have,
\begin{equation}
{}_0^C\mathrm{D}_t^{\frac{1}{2}} {}_0^C\mathrm{D}_t^{\frac{1}{2}}E_\frac{1}{2}(\lambda t^\frac{1}{2})= \lambda^2 E_\frac{1}{2}(\lambda t^\frac{1}{2}) \label{2.15}
\end{equation}
and
\begin{equation}
\mathrm{D}E_\frac{1}{2}(\lambda t^\frac{1}{2})= \lambda t^{-\frac{1}{2}} E_{\frac{1}{2},\frac{1}{2}}(\lambda t^\frac{1}{2}). \label{2.16}
\end{equation}
\begin{equation*}
\Rightarrow \,{}_0^C\mathrm{D}_t^{\frac{1}{2}} {}_0^C\mathrm{D}_t^{\frac{1}{2}}E_\frac{1}{2}(\lambda t^\frac{1}{2})\ne \mathrm{D}E_\frac{1}{2}(\lambda t^\frac{1}{2}).
\end{equation*}
\end{Ex}
\begin{Ex}
We have,
\begin{equation}
{}_0^C\mathrm{D}_t^{\frac{3}{10}} {}_0^C\mathrm{D}_t^{\frac{3}{10}}E_\frac{3}{10}(\lambda t^\frac{3}{10})= \lambda^2 E_\frac{3}{10}(\lambda t^\frac{3}{10}) \label{2.17}
\end{equation}
and
\begin{equation}
{}_0^C\mathrm{D}_t^{\frac{6}{10}}E_\frac{3}{10}(\lambda t^\frac{3}{10})= \lambda t^{-\frac{3}{10}} E_{\frac{3}{10},\frac{7}{10}}(\lambda t^\frac{3}{10}). \label{2.18}
\end{equation}
\begin{equation*}
\Rightarrow \, {}_0^C\mathrm{D}_t^{\frac{3}{10}} {}_0^C\mathrm{D}_t^{\frac{3}{10}}E_\frac{3}{10}(\lambda t^\frac{3}{10})\ne {}_0^C\mathrm{D}_t^{\frac{6}{10}}E_\frac{3}{10}(\lambda t^\frac{3}{10}).
\end{equation*}
\end{Ex}
\end{remark}
{\bf Note 2:}\\
The condition given in Theorem \ref{Thm 2.4} is sufficient but not necessary.\\
Let, \,
 $x(t)=E_\alpha(\lambda t^\alpha),\,\, where \,\, 0<\alpha<1.$\\
 Let us consider $\alpha_1\in \mathbb{R}$ such that $0<\alpha_1<\alpha<1$.\\ 
 \begin{equation*}
 \begin{split}
\mathrm{Therefore} \quad  {}_0^C\mathrm{D}_t^{\alpha_1}x(t) & = {}_0^C\mathrm{D}_t^{\alpha_1} \left\{\sum_{k=0}^\infty \frac{(\lambda t^\alpha)^k}{\Gamma(\alpha k+1)}\right\}\\
& = \sum_{k=1}^\infty \frac{\lambda^k t^{\alpha k-\alpha_1}}{\Gamma(\alpha k-\alpha_1+1)}.
 \end{split}
 \end{equation*}
 ${}_0^C\mathrm{D}_t^{\alpha_1}x(t)|_{t=0}=0$,\quad 
  ($\because \alpha-\alpha_1>0$).\\
 However, $x(t)\notin C^1[0,T]$ for any $T>0$.

\section{Equivalence between FDE and a system of lower order FDEs obtained by splitting the original equation }
In this section we provide the conditions for the equivalence between higher order FDE and a system of FDEs of lower order obtained by splitting the original one.
\subsection{One term case:}
Consider linear FDE,
\begin{equation}
{}_0^C\mathrm{D}_t^\alpha x(t)=\lambda x(t), \quad 0<\alpha<1. \label{2.19}
\end{equation}
Its general solution is,
\begin{equation}
x(t)=x(0)E_\alpha(\lambda t^\alpha). \label{2.20}
\end{equation}
If $\alpha_1+\alpha_2=\alpha$ and if we could  write 
\begin{equation}
{}_0^C\mathrm{D}_t^{\alpha_2} {}_0^C\mathrm{D}_t^{\alpha_1}x(t)=\lambda x(t), \label{2.21}
\end{equation}
then (\ref{2.19}) is equivalent to a system
\begin{equation}
\begin{split}
{}_0^C\mathrm{D}_t^{\alpha_1}x(t)& =y(t), \\
{}_0^C\mathrm{D}_t^{\alpha_2}y(t)& =\lambda x(t)\\
\mathrm{or} \quad {}_0^C\mathrm{D}_t^{\alpha-\alpha_1}y(t)& =\lambda x(t).
\end{split} \label{2.22}
\end{equation}
 Laplace transform of system (\ref{2.22}) gives
\begin{equation*}
\begin{split}
s^{\alpha_1}X(s)-s^{\alpha_1-1}x(0)& =Y(s)\quad \mathrm{and}\\
s^{\alpha-\alpha_1}Y(s)-s^{\alpha-\alpha_1-1}y(0)& =\lambda X(s).
\end{split}
\end{equation*}
Solving this system for $X(s)$, we get
\begin{equation*}
 X(s) =x(0)\frac{s^{\alpha-1}}{s^\alpha-\lambda}+y(0)\frac{s^{\alpha-\alpha_1-1}}{s^\alpha-\lambda}.
\end{equation*}
 Inverse Laplace transform gives
\begin{equation*}
 x(t) =x(0)E_\alpha(\lambda t^\alpha)+y(0)t^{\alpha_1}E_{\alpha,\alpha_1+1}(\lambda t^\alpha).
\end{equation*}
By Note 2, we have $y(0)={}_0^C\mathrm{D}_t^{\alpha_1}x(0)=0$.
Thus, solution of the system (\ref{2.22}) is,
\begin{equation}
x(t) =x(0)E_\alpha(\lambda t^\alpha). \label{2.23}
\end{equation}
Thus the equation (\ref{2.19}) is equivalent to the system (\ref{2.22}).
\subsection{Two term case:}
Consider linear 2-term FDE,
\begin{equation}
{}_0^C\mathrm{D}_t^\alpha x(t)+a_1{}_0^C\mathrm{D}_t^\beta x(t)+a_2x(t)=0, \quad 0<\beta<1,\,1<\alpha<2. \label{2.24}
\end{equation}
Laplace transform of the equation (\ref{2.24}) is,
\begin{equation}
 X(s) =x(0)\frac{s^{\alpha-1}+a_1s^{\beta-1}}{s^\alpha+a_1s^\beta+a_2}+x'(0)\frac{s^{\alpha-2}}{s^\alpha+a_1s^\beta+a_2}. \label{2.25}
\end{equation}
Using inverse Laplace transform, we obtain the solution of (\ref{2.24}) as below:
\begin{equation}
\begin{split}
 x(t)& =x(0)\left[E_{(\alpha-\beta, \alpha),1}(-a_1t^{\alpha-\beta},-a_2t^\alpha)+a_1t^{\alpha-\beta}E_{(\alpha-\beta, \alpha),\alpha-\beta+1}(-a_1t^{\alpha-\beta},-a_2t^\alpha)\right]\\
& \quad +x'(0)tE_{(\alpha-\beta, \alpha),2}(-a_1t^{\alpha-\beta},-a_2t^\alpha).
\end{split} \label{2.26}
\end{equation} 
Now, consider the system
\begin{equation}
\begin{split}
{}_0^C\mathrm{D}_t^{\beta}x(t)& =y(t), \\
{}_0^C\mathrm{D}_t^{\alpha-\beta}y(t)& =-a_2 x(t)-a_1 y(t)
\end{split} \label{2.27}
\end{equation}
obtained by splitting the FDE (\ref{2.24}).\\
As $1<\alpha\le2$, so we can write $\alpha=1+\delta$, where $0<\delta\le 1$.\\
 We consider following three subcases  and check whether the system (\ref{2.27}) is equivalent to (\ref{2.24}).\\[0.15cm]
{\bf Case (i):}\\
Suppose that  $0<\beta<\delta\le1$.\quad
$\Rightarrow 1<\alpha-\beta<2$.\\
Now, Laplace transform of the system (\ref{2.27}) gives,
\begin{equation}
\begin{split}
X(s) & =x(0)\frac{s^{\alpha-1}+a_1s^{\beta-1}}{s^\alpha+a_1s^\beta+a_2}+y(0)\frac{s^{\alpha-\beta-1}}{s^\alpha+a_1s^\beta+a_2}
\\ & \quad +y'(0)\frac{s^{\alpha-\beta-2}}{s^\alpha+a_1s^\beta+a_2}.
\end{split} \label{2.28}
\end{equation}
If the equation (\ref{2.24}) is equivalent to the system (\ref{2.27}) then their Laplace transforms (\ref{2.25}) and (\ref{2.28}) will match.\\
Comparing (\ref{2.25}) and (\ref{2.28}), we get\\
$$y(0)s^{-1}-y'(0)s^{-2}=x'(0)s^{\beta-2}$$
\begin{equation}
\Rightarrow y(0)s-y'(0)=x'(0)s^\beta.   \label{2.29}
\end{equation}
Since $0<\beta<1$, the equation (\ref{2.29}) cannot be solved for $y(0)$ and $y'(0)$ in terms of $x'(0)$. This shows that the Laplace transforms (\ref{2.25}) and (\ref{2.28}) are different.\\
$\Rightarrow$
In this case equation (\ref{2.24}) is not equivalent to the system (\ref{2.27}).
\\[0.15cm]

{\bf Case (ii):}\\
Suppose that $0<\delta<\beta<1$.\quad
$\Rightarrow 0<\alpha-\beta<1$.\\
Using Laplace transform to the system (\ref{2.27}), we obtain
\begin{equation}
X(s)  =x(0)\left(\frac{s^{\alpha-1}}{s^\alpha+a_1s^\beta+a_2}+a_1\frac{s^{\beta-1}}{s^\alpha+a_1s^\beta+a_2}\right)+y(0)\frac{s^{\alpha-\beta-1}}{s^\alpha+a_1s^\beta+a_2}. \label{2.30}
\end{equation}
 Comparing (\ref{2.25}) and (\ref{2.30}), we get $\frac{x'(0)}{y(0)}=s^{1-\beta}$. This contradiction proves that the equation (\ref{2.24}) is not equivalent to the system (\ref{2.27}) in this case also.\\[0.15cm]
{\bf Case (iii):}\\
Consider $0<\beta=\delta\le1$.\quad
$\Rightarrow \alpha-\beta=1$.\\
The equation (\ref{2.24}) is not equivalent to the system (\ref{2.27}) because ${}_0^{C}\mathrm{D}_t^m \mathrm{D}_t^\alpha\ne {}_0^{C}\mathrm{D}_t^{m+\alpha}, \, m=0,1,\dots$.\\[0.2cm]
Now, we provide a proper way to split FDE (\ref{2.24}).
\subsubsection{Proper way of splitting two term FDE:} Let us split FDE (\ref{2.24}) in the following system of three FDEs. 
\begin{equation}
\begin{split}
{}_0^C\mathrm{D}_t^{\beta}x(t)& =y_1(t),\\
{}_0^C\mathrm{D}_t^{1-\beta}y_1(t)& = y_2(t),\\
{}_0^C\mathrm{D}_t^{\alpha-1}y_1(t)& =-a_2 x(t)-a_1 y_1(t)
\end{split} \label{2.31}
\end{equation}
with initial conditions
$$y_1(0)=0, y_2(0)=x'(0) $$
The Laplace transform of system (\ref{2.31}) gives
\begin{equation*}
\begin{split}
X(s) & =x(0)\frac{s^{\alpha-1}+a_1s^{\beta-1}}{s^\alpha+a_1s^\beta+a_2}+y_1(0)\frac{s^{\alpha-\beta-1}}{s^\alpha+a_1s^\beta+a_2}
\\ & \quad +y_2(0)\frac{s^{\alpha-2}}{s^\alpha+a_1s^\beta+a_2}\\
& = x(0)\frac{s^{\alpha-1}+a_1s^{\beta-1}}{s^\alpha+a_1s^\beta+a_2}+x'(0)\frac{s^{\alpha-2}}{s^\alpha+a_1s^\beta+a_2}.
\end{split} 
\end{equation*}
This expression is same as (\ref{2.25}).\\
Therefore, the FDE (\ref{2.24}) is equivalent to the system (\ref{2.31}).
\subsection{Multi-term case:}
The following Theorem gives the proper way of splitting multi-term FDE into a system of FDEs with lower orders.
\begin{The} \label{Thm 2.10}
	Consider multi-term FDE 
	\begin{equation}
a_0x(t)+a_1{}_0^C\mathrm{D}_t^{\alpha_1}x(t)+a_2{}_0^C\mathrm{D}_t^{\alpha_2}x(t)+\dots+ a_m{}_0^C\mathrm{D}_t^{\alpha_m}x(t)=0, \label{2.32}
	\end{equation}
	where $k-1<\alpha_k\le k$,\quad  k=1,2,\dots,m \,\,subject
	to initial conditions\\ $x^{(i)}(0)=C_i, \quad i=0,1,\dots,m-1$.\\
	If we split equation (\ref{2.32}) in a system of $2m-1$ equations as
	\begin{equation}
	\left.
	\begin{split}
	x(t) & =y_0(t),\\
	{}_0^C\mathrm{D}_t^{\beta_j}y_j(t) & =y_{j+1}(t),\quad j=0,1,2,\dots,2m-3,\\
	{}_0^C\mathrm{D}_t^{\beta_{2m-2}}y_{2m-2}(t) & = \frac{-1}{a_m}[a_0x(t)+a_1y_1(t)+a_2y_3(t)\\
	&\quad+a_3y_5(t)+\dots+a_{m-1}y_{2m-3}(t)],
	\end{split}
	\right\} \label{2.33}
	\end{equation}
	where, 
	\begin{equation}
	\begin{split}
	\beta_{2j} & =\alpha_{j+1}-j, \qquad j=0,1,2,\dots,m-1,\\
	\beta_{2j+1} & =(j+1)-\alpha_{j+1}, \qquad j=0,1,2,\dots,m-2,	
  \end{split} \label{2.34}
  \end{equation}
  and\,\, $0<\beta_k\le 1$, \quad $k=0,1,2,\dots,2m-2$ and
with initial conditions 
\begin{equation}
\begin{split}
y_{2i-1}(0) &= 0, \qquad \qquad \qquad \,\,\, i=1,2,\dots,m-1,\\
y_{2i}(0) &= x^{(i)}(0)=C_i, \qquad i=1,2,\dots,m-1
 \label{2.35}
\end{split} 
\end{equation}
	then the equation (\ref{2.32}) is equivalent to the system (\ref{2.33}).
\end{The}
\begin{proof}
Laplace transform of equation (\ref{2.32}) is given by,
\begin{equation}
X(s)=\frac{1}{a_0+\sum_{i=1}^m a_is^{\alpha_i}}\left\{\sum_{i=0}^{m-1}x^{(i)}(0)\left( \sum_{j=i+1}^m a_js^{\alpha_j-i-1}\right) \right\}. \label{2.36}
\end{equation}
Similarly, the Laplace transform of the system (\ref{2.33}) is
\begin{equation*}
\begin{bmatrix}
s^{\beta_0} & -1 & 0 & 0 & \dots & 0 & 0\\
0 & s^{\beta_1} & -1 & 0 & \dots & 0 & 0\\
0 & 0 & s^{\beta_2} & -1 & \dots & 0 & 0\\
0 & 0 & 0 & s^{\beta_3} & \dots & 0 & 0\\
\hdotsfor{7}\\
 0 & 0 & 0 & 0 & \dots & s^{\beta_{2m-3}} & -1\\
 \frac{a_0}{a_m} & \frac{a_1}{a_m} & 0 & \frac{a_2}{a_m} & \dots & \frac{a_{m-1}}{a_m} & s^{\beta_{2m-2}}
\end{bmatrix}
\begin{bmatrix}
Y_0(s)\\
Y_1(s)\\
Y_2(s)\\
Y_3(s)\\
\vdots\\
Y_{2m-3}(s)\\
Y_{2m-2}(s)\\
\end{bmatrix}
=
\begin{bmatrix}
s^{\beta_0-1}y_0(0)\\
s^{\beta_1-1}y_1(0)\\
s^{\beta_2-1}y_2(0)\\
s^{\beta_3-1}y_3(0)\\
\vdots\\
s^{\beta_{2m-3}-1}y_{2m-3}(0)\\
s^{\beta_{2m-2}-1}y_{2m-2}(o)\\
\end{bmatrix}.
\end{equation*}
\begin{equation}
\begin{split}
\Rightarrow \,Y_0(s) & =\frac{\sum_{i=0}^{m-1}y_{2i}(0)\left( \sum_{j=i+1}^m a_{m-j+i+1}s^{\beta_{2i}+\beta_{2i+1}+\dots+\beta_{2i+2m-2j}-1}\right)}{a_0+\sum_{i=1}^m a_is^{\sum_{k=0}^{2i-2}\beta_k}} \\
& \quad +\frac{\sum_{i=1}^{m-1}y_{2i-1}(0)\left( \sum_{j=i+1}^m a_{m-j+i+1}s^{\beta_{2i-1}+\beta_{2i}+\dots+\beta_{2i+2m-2j}-1}\right)}{a_0+\sum_{i=1}^m a_is^{\sum_{k=0}^{2i-2}\beta_k}}.
\end{split} \label{2.37}
\end{equation}
Substituting (\ref{2.34}) and using initial conditions (\ref{2.35}) in equation (\ref{2.37}) we can conclude that	
 $$Y_0(s)=X(s).$$ 
$\Rightarrow$ The FDE (\ref{2.32}) is equivalent to the system (\ref{2.33}). This completes the proof.
\end{proof}
\begin{remark}
(i) Any system with equations more than 2m-1  and obtained by splitting (\ref{2.33}) is equivalent to (\ref{2.32}).\\
(ii) Splitting equation (\ref{2.32}) in a system containing  equations less than 2m-1 is not allowed.
\end{remark}
The following Theorem \ref{Thm 2.11} illustrates this result for the case of $2m-2$ equations.
\begin{The} \label{Thm 2.11}
If we split the equation (\ref{2.32}) in a system involving $2m-2$ equations as
\begin{equation}
\left.
\begin{split}
x(t) & =y_0(t),\\
{}_0^C\mathrm{D}_t^{\beta_j}y_j(t) & =y_{j+1}(t),\quad j=0,1,2,\dots,2m-4\\
{}_0^C\mathrm{D}_t^{\beta_{2m-3}}y_{2m-3}(t) & = \frac{-1}{a_m}[a_0y_0(t)+a_1y_1(t)+a_2y_3(t)\\
&\quad +a_3y_5(t)+\dots+a_{m-1}y_{2m-3}(t)],
\end{split}
\right\} \label{2.38}
\end{equation}
with
\begin{equation}
\begin{split}
\beta_0& =\alpha_1,\\
\alpha_1+\beta_1+\beta_2& =\alpha_2,\\
\alpha_2+\beta_3+\beta_4& =\alpha_3,\\
\alpha_3+\beta_5+\beta_6& =\alpha_4,\\
& \vdots\\
\alpha_{m-2}+\beta_{2m-5}+\beta_{2m-4}& =\alpha_{m-1},\\
\alpha_{m-1}+\beta_{2m-3}& =\alpha_m .
\end{split} \label{2.39}
\end{equation}
then the equation (\ref{2.32})	is not equivalent to the system (\ref{2.38}).
\end{The}
\begin{proof}
The Laplace transform of system (\ref{2.38}) gives the system
\begin{equation*}
\begin{bmatrix}
s^{\beta_0} & -1 & 0 & 0 & \dots & 0 & 0\\
0 & s^{\beta_1} & -1 & 0 & \dots & 0 & 0\\
0 & 0 & s^{\beta_2} & -1 & \dots & 0 & 0\\
0 & 0 & 0 & s^{\beta_3} & \dots & 0 & 0\\
\hdotsfor{7}\\
 0 & 0 & 0 & 0 & \dots & s^{\beta_{2m-4}} & -1\\
 \frac{a_0}{a_m} & \frac{a_1}{a_m} & 0 & \frac{a_2}{a_m} & \dots & 0 & \frac{a_{m-1}}{a_m}+s^{\beta_{2m-3}}
\end{bmatrix}
\begin{bmatrix}
Y_0(s)\\
Y_1(s)\\
Y_2(s)\\
Y_3(s)\\
\vdots\\
Y_{2m-4}(s)\\
Y_{2m-3}(s)\\
\end{bmatrix}
=
\begin{bmatrix}
s^{\beta_0-1}y_0(0)\\
s^{\beta_1-1}y_1(0)\\
s^{\beta_2-1}y_2(0)\\
s^{\beta_3-1}y_3(0)\\
\vdots\\
s^{\beta_{2m-4}-1}y_{2m-4}(0)\\
s^{\beta_{2m-3}-1}y_{2m-3}(o)
\end{bmatrix}.
\end{equation*}
\begin{equation}
\begin{split}
\therefore Y_0(s)& =\frac{\sum_{i=0}^{m-2}y_{2i}(0)\left( \sum_{j=i+1}^{m-1} a_{m-j+i}s^{\sum_{k=2i}^{2(m-j+i)-2} \beta_k-1}+a_ms^{\sum_{k=2i}^{2m-3}\beta_k-1}\right)}{\sum_{j=0}^{m-1}a_js^{\sum_{k=0}^{2j-2}\beta_k}+a_ms^{\sum_{k=0}^{2m-3}\beta_k}}\\
& \quad + \frac{\sum_{i=1}^{m-1}y_{2i-1}(0)\left( \sum_{j=i+1}^{m-1} a_{m-j+i}s^{\sum_{k=2i-1}^{2(m-j+i)-2} \beta_k-1}+a_ms^{\sum_{k=2i-1}^{2m-3}\beta_k-1}\right)}{\sum_{j=0}^{m-1}a_js^{\sum_{k=0}^{2j-2}\beta_k}+a_ms^{\sum_{k=0}^{2m-3}\beta_k}}.
\end{split} \label{2.40}
\end{equation}
Using (\ref{2.39}), we get
\begin{equation}
\begin{split}
Y_0(s) & =\frac{\sum_{i=0}^{m-2}y_{2i}(0)\left( \sum_{j=i+1}^m a_js^{\alpha_j-\alpha_i-\beta_{2i-1}-1}\right)}{a_0+\sum_{i=1}^m a_is^{\alpha_i}} \\
& \quad +\frac{\sum_{i=1}^{m-1}y_{2i-1}(0)\left( \sum_{j=i+1}^m a_js^{\alpha_j-\alpha_i-1}\right)}{a_0+\sum_{i=1}^m a_is^{\alpha_i}}.
\end{split} \label{2.41}
\end{equation} 
If (\ref{2.32}) is equivalent to the system (\ref{2.38}), then 
$X(s)=Y_0(s)$.\\
Comparing equations (\ref{2.36}) and (\ref{2.41}), we get
\begin{eqnarray}
y_{2i-1}(0)= 0, \quad \quad i=1,2,\dots,m-1,\nonumber\\
y_{2i}(0) = x^{(i)}(0) \quad \mathrm{and}\quad \beta_{2i-1} = i-\alpha_i,\quad \quad i=1,2,\dots,m-2.\nonumber
\end{eqnarray}
However, there is no term in equation (\ref{2.41}) which matches with the term $x^{(m-1)}(0)a_ms^{\alpha_m-m}$ appearing in the Laplace transform $X(s)$ given by (\ref{2.36}).\\
Hence (\ref{2.32}) is not equivalent to the system (\ref{2.38}).
\end{proof}
\begin{The}\label{Thm 2.12}
	Consider 
	\begin{equation}
a_0x(t)+a_1{}_0^C\mathrm{D}_t^{\alpha_1}x(t)+a_2{}_0^C\mathrm{D}_t^{\alpha_2}x(t)+\dots a_m{}_0^C\mathrm{D}_t^{\alpha_m}x(t)=0, \label{2.42}
	\end{equation}
	where $p<\alpha_1<\alpha_2<\dots<\alpha_m\le p+1$,\quad  $p \in \mathbb{N}\cup\{0\}$,\\
	 with initial conditions $x^{(i)}(0)=C_i, \qquad i=0,1,2,\dots,p$.\\
	If we split equation (\ref{2.42}) as
	\begin{equation}
	\begin{split}
	{}_0^C\mathrm{D}_t^{\alpha_1}x(t)& =y_1(t),\\
	{}_0^C\mathrm{D}_t^{\alpha_{j+1}-\alpha_j}y_j(t) & =y_{j+1}(t),\quad j=1,2,\dots,m-2,\\
	{}_0^C\mathrm{D}_t^{\alpha_m-\alpha_{m-1}}y_{m-1}(t) & = \frac{-1}{a_m}[a_0x(t)+a_1y_1(t)+a_2y_2(t)\\
	& \quad +a_3y_3(t)+\dots+a_{m-1}y_{m-1}(t)]
	\end{split} \label{2.43}
	\end{equation}
	with initial conditions $y_j(0)=0, \qquad j=1,2,\dots,m-1$,\\
	then the equation (\ref{2.42}) is equivalent to the system (\ref{2.43}).
\end{The}
\begin{proof}
The Laplace transform of the equation (\ref{2.42}) is,
\begin{equation}
X(s)=\frac{1}{a_0+\sum_{k=1}^m a_ks^{\alpha_k}}\left\{\sum_{i=0}^px^{(i)}(0)\left( \sum_{j=i+1}^m a_js^{\alpha_j-i-1}\right) \right\}. \label{2.44}
\end{equation}
Similarly, the Laplace transform of the system (\ref{2.43}) is,
\begin{equation*}
\begin{bmatrix}
s^{\alpha_1} & -1 & 0 &  \dots & 0 & 0\\
0 & s^{\alpha_2-\alpha_1} & -1 & \dots & 0 & 0\\
0 & 0 & s^{\alpha_3-\alpha_2} & \dots & 0 & 0\\
\hdotsfor{6}\\
 0 & 0 & 0 &  \dots & s^{\alpha_{m-1}-\alpha_{m-2}} & -1\\
 \frac{a_0}{a_m} & \frac{a_1}{a_m} & \frac{a_2}{a_m}  & \dots & \frac{a_{m-2}}{a_m} & \frac{a_{m-1}}{a_m}s^{\alpha_m-\alpha_{m-1}}
\end{bmatrix}
\begin{bmatrix}
X(s)\\
Y_1(s)\\
Y_2(s)\\
\vdots\\
Y_{m-2}(s)\\
Y_{m-1}(s)\\
\end{bmatrix}
=
\begin{bmatrix}
\sum_{i=0}^ps^{\alpha_1-i-1}x^{(i)}(0)\\
s^{\alpha_2-\alpha_1-1}y_1(0)\\
s^{\alpha_3-\alpha_2-1}y_2(0)\\
\vdots\\
s^{\alpha_{m-1}-\alpha_{m-2}-1}y_{m-2}(0)\\
s^{\alpha_m-\alpha_{m-1}-1}y_{m-1}(o)\\
\end{bmatrix}.
\end{equation*}
Simplifying, we get
\begin{equation}
\begin{split}
 X(s)& =\sum_{i=0}^px^{(i)}(0)\frac{\left( \sum_{j=i+1}^m a_js^{\alpha_j-i-1}\right)}{a_0+\sum_{k=1}^m a_ks^{\alpha_k}} \\
& \quad + \sum_{i=1}^{m-1}y_i(0)\frac{\left( \sum_{j=i+1}^m a_js^{\alpha_j-\alpha_i-1}\right)}{a_0+\sum_{k=1}^m a_ks^{\alpha_k}}.
\end{split}\label{2.45}
\end{equation}
Using initial conditions $y_i(0)=0$, $i=1,2,\dots,m-1$ in (\ref{2.45}), we get equation (\ref{2.44}).
This proves the result.
\section{Conclusion}
In this article, we discussed the conditions for compositions of Caputo fractional derivatives. We have shown that some Mittag-Leffler functions are satisfying these rules even though the conditions in the respective Theorems given in the literature are not satisfied. Further, we proposed the results describing the proper ways to split the linear FDEs into the systems of FDEs involving lower order derivatives. 
\section{Acknowledgment}
S. Bhalekar acknowledges  the Science and Engineering Research Board (SERB), New Delhi, India for the Research Grant (Ref. MTR/2017/000068) under Mathematical Research Impact Centric Support (MATRICS) Scheme. M. Patil acknowledges Department of Science and Technology (DST), New Delhi, India for INSPIRE Fellowship (Code-IF170439).
\end{proof}
    
\end{document}